\documentclass[11pt]{article}

\usepackage{amssymb,amsmath,amsfonts,amsthm}
\usepackage{latexsym}
\usepackage{graphics}
\usepackage{indentfirst}

\newtheorem*{thm*}{Theorem}
\newtheorem{thm}{Theorem}

\newtheorem{cor}[thm]{Corollary}

\newcommand{\N}{\mathbb{N}}
\newcommand{\Z}{\mathbb{Z}}

\newcommand{\col}{\mathrm{col}}

\begin{document}

\title{On the Alon-Tarsi Number and Chromatic-choosability of Cartesian Products of Graphs}

\author{Hemanshu Kaul\footnote{Department of Applied Mathematics, Illinois Institute of Technology, Chicago, IL 60616. E-mail: {\tt kaul@iit.edu}} $\;$ and Jeffrey A. Mudrock\footnote{Department of Applied Mathematics, Illinois Institute of Technology, Chicago, IL 60616. E-mail: {\tt jmudrock@hawk.iit.edu}} }

\date{}

\maketitle

\begin{abstract}
We study the list chromatic number of Cartesian products of graphs through the Alon-Tarsi number as defined by Jensen and Toft (1995) in their seminal book on graph coloring problems. The \emph{Alon-Tarsi number} of $G$, $AT(G)$, is the smallest $k$ for which there is an orientation, $D$, of $G$ with max indegree $k\!-\!1$ such that the number of even and odd circulations contained in $D$ are different. It is known that $\chi(G) \leq \chi_\ell(G) \leq \chi_p(G) \leq AT(G)$, where  $\chi(G)$ is the chromatic number, $\chi_\ell(G)$ is the list chromatic number, and $\chi_p(G)$ is the paint number of $G$. In this paper we find families of graphs $G$ and $H$ such that $\chi(G \square H) = AT(G \square H)$, reducing this sequence of inequalities to equality.

We show that the Alon-Tarsi number of the Cartesian product of an odd cycle and a path is always equal to 3. This result is then extended to show that if $G$ is an odd cycle or a complete graph and $H$ is a graph on at least two vertices containing the Hamilton path $w_1, w_2, \ldots, w_n$ such that for each $i$, $w_i$ has a most $k$ neighbors among $w_1, w_2, \ldots, w_{i-1}$, then $AT(G \square H) \leq \Delta(G)+k$ where $\Delta(G)$ is the maximum degree of $G$.  We discuss other extensions for $G \square H$, where $G$ is such that $V(G)$ can be partitioned into odd cycles and complete graphs, and $H$ is a graph containing a Hamiltonian path. We apply these bounds to get chromatic-choosable Cartesian products, in fact we show that these families of graphs have $\chi(G) = AT(G)$, improving previously known bounds.
\medskip



\end{abstract}

\section{Introduction}\label{intro}

In this paper all graphs are finite, and all graphs are either simple graphs or simple directed graphs.  Generally speaking we follow West~\cite{W01} for terminology and notation.  List coloring is a well known variation on the classic vertex coloring problem, and it was introduced independently by Vizing~\cite{V76} and Erd\H{o}s, Rubin, and Taylor~\cite{ET79} in the 1970's.  In the classic vertex coloring problem we wish to color the vertices of a graph $G$ with as few colors as possible so that adjacent vertices receive different colors, a so-called \emph{proper coloring}. The chromatic number of a graph, denoted $\chi(G)$, is the smallest $k$ such that $G$ has a proper coloring that uses $k$ colors.  For list coloring, we associate a \emph{list assignment}, $L$, with a graph $G$ such that each vertex $v \in V(G)$ is assigned a list of colors $L(v)$ (we say $L$ is a list assignment for $G$).  The graph $G$ is \emph{$L$-colorable} if there exists a proper coloring $f$ of $G$ such that $f(v) \in L(v)$ for each $v \in V(G)$ (we refer to $f$ as a \emph{proper $L$-coloring} for $G$). The \emph{list chromatic number} of a graph $G$, denoted $\chi_\ell(G)$, is the smallest $k$ such that $G$ is $L$-colorable whenever the list assignment $L$ satisfies $|L(v)| \geq k$ for each $v \in V(G)$. It is immediately obvious that for any graph $G$, $\chi(G) \leq \chi_\ell(G)$.  Erd\H{o}s, Taylor, and Rubin observed in~\cite{ET79} that bipartite graphs can have arbitrarily large list chromatic number. This means that the gap between $\chi(G)$ and $\chi_\ell(G)$ can be arbitrarily large, and we can not hope to find an upper bound for $\chi_\ell(G)$ in terms of just $\chi(G)$.

Graphs in which $\chi(G) = \chi_\ell(G)$ are known as \emph{chromatic-choosable} graphs (see ~\cite{O02}). Many classes of graphs have been conjectured to be chromatic-choosable. The most well known conjecture along these lines is the List Coloring Conjecture (see~\cite{HC92}) which states that every line graph of a loopless multigraph is chromatic-choosable. In addition, total graphs (\cite{BKW97}) and claw free graphs (\cite{GM97}) are conjectured to be chromatic-choosable. On the other hand, there are classes of graphs that are known to be chromatic-choosable. In 1995, Galvin~\cite{G95} showed that the List Coloring Conjecture holds for line graphs of bipartite multigraphs, and in 1996, Kahn~\cite{K96} proved an asymptotic version of the conjecture. Tuza and Voigt~\cite{TV96}  showed that chordal graphs are chromatic-choosable, and Prowse and Woodall~\cite{PW03} showed that powers of cycles are chromatic-choosable.  Recently, Noel, Reed, and Wu~\cite{NR15} proved Ohba's conjecture which states that every graph, $G$, on at most $2 \chi(G) + 1$ vertices is chromatic-choosable.

%Now, we mention some basic upper bounds on the list chromatic number. The \emph{coloring number} of a graph $G$, denoted $\col(G)$, is the smallest integer $d$ for which there exists an ordering, $v_1, v_2, \ldots, v_n$, of the elements in $V(G)$ such that each vertex $v_i$ has at most $d-1$ neighbors among $v_1, v_2, \ldots, v_{i-1}$. For any graph $G$, we have that $\chi_\ell(G) \leq \col(G)$, since we can use the vertex ordering to greedily find a proper $L$-coloring for $G$ whenever the list assignment $L$ satisfies $|L(v)| \geq \col(G)$ for each $v \in V(G)$. Let $\Delta(G)$ denote the maximum degree of the vertices in a graph $G$.  Clearly, $\col(G) \leq \Delta(G) + 1$. Vizing~\cite{V76}  proved the following list coloring version of Brooks' Theorem.

%\begin{thm}[\cite{V76}] \label{thm: Brooks}
%Suppose that $G$ is a connected graph with maximum degree $\Delta(G)$.  If $G$ is neither a complete graph nor an odd cycle, then $\chi_\ell(G) \leq \Delta(G)$.
%\end{thm}

In this paper, we continue this investigation of chromatic-choosability in the realm of Cartesian Products of graphs. We study the list chromatic number of Cartesian products of graphs through the Alon-Tarsi number as defined by Jensen and Toft in their seminal 1995 book on graph coloring problems~\cite{JT95}. The \emph{Alon-Tarsi number} (AT-number for short) of $G$, $AT(G)$, is the smallest $k$ for which there is an orientation, $D$, of $G$ with max indegree $k-1$ such that the number of even and odd circulations contained in $D$ are different. It follows from the Alon-Tarsi Theorem~\cite{AN92} that $\chi(G) \leq \chi_\ell(G) \leq AT(G)$. We are interested in finding $G$ for which these three parameters are equal.

In the next two subsections, we discuss the known bounds for the list chromatic number of the Cartesian product of graphs as well as the main tool we use to obtain our results: the Alon-Tarsi Theorem. In Section~2, we give a series of sharp bounds on the AT-number of Cartesian Products of the form $G \square H$, where $G$ is such that $V(G)$ can be partitioned into odd cycles and complete graphs, and $H$ is a traceable graph (a graph containing a Hamiltonian path). In Section~3, we apply the bounds from the previous section to give some examples of chromatic-choosable Cartesian products of graphs, or even more strongly we show that these graphs have $\chi(G) = \chi_\ell(G) = AT(G)$.

\subsection{Cartesian Product of Graphs}\label{prelim}

The \emph{Cartesian product} of graphs $G$ and $H$, denoted $G \square H$, is the graph with vertex set $V(G) \times V(H)$ and edges created so that $(u,v)$ is adjacent to $(u',v')$ if and only if either $u=u'$ and $vv' \in E(H)$ or $v=v'$ and $uu' \in E(G)$. Note that $G \square H$ contains $|V(G)|$ copies of $H$ and $|V(H)|$ copies of $G$.  It is also easy to show that $\chi(G \square H) = \max \{\chi(G), \chi(H) \}$.  So, we have that $\max \{\chi(G), \chi(H) \} \leq \chi_\ell(G \square H)$.

\par

There are few results in the literature regarding the list chromatic number of the Cartesian product of graphs. In 2006, Borowiecki, Jendrol, Kr{\'a}l, and Mi{\v s}kuf~\cite{BJ06} showed the following.
\begin{thm}[\cite{BJ06}] \label{thm: Borow}
$\chi_\ell(G \square H) \leq \min \{\chi_\ell(G) + \col(H), \col(G) + \chi_\ell(H) \} - 1.$
\end{thm}

Here $\col(G)$, the \emph{coloring number} of a graph $G$, is the smallest integer $d$ for which there exists an ordering, $v_1, v_2, \ldots, v_n$, of the elements in $V(G)$ such that each vertex $v_i$ has at most $d-1$ neighbors among $v_1, v_2, \ldots, v_{i-1}$. The coloring number is a classic greedy upper bound on the list chromatic number, and it immediately implies that $\Delta(G)+1$ is an upper bound on the list chromatic number where $\Delta(G)$ is the maximum degree of $G$. Vizing~\cite{V76} extended this by proving the list coloring version of Brooks' Theorem.

\begin{thm}[\cite{V76}] \label{thm: Brooks}
Suppose that $G$ is a connected graph with maximum degree $\Delta(G)$.  If $G$ is neither a complete graph nor an odd cycle, then $\chi_\ell(G) \leq \Delta(G)$.
\end{thm}

Borowiecki et al.~\cite{BJ06} construct examples where the upper bound in their theorem is tight.  Specifically, they show that if $k \in \N$ and $G$ is a copy of the complete bipartite graph $K_{k, (2k)^{k(k+k^k)}}$, then $\chi_\ell(G \square G)=\chi_\ell(G)+\col(G)-1$.  On the other hand, there are examples where the upper bounds from Theorems~\ref{thm: Borow} and~\ref{thm: Brooks} are not tight. For example, suppose that $G$ is a copy of $C_{2k+1} \square P_n$ where $n \geq 3$.  Since $\chi_\ell(C_{2k+1})=\col(C_{2k+1})=3$ and $\chi_\ell(P_n)=\col(P_n)=2$, Theorem~\ref{thm: Borow} tells us that $\chi_\ell(G) \leq 4$.  Similarly, Theorem~\ref{thm: Brooks} tells us $\chi_\ell(G) \leq 4$, yet we will show below that $\chi_\ell(G)=3$.  Another, more dramatic, example where Theorems~\ref{thm: Borow} and~\ref{thm: Brooks} do not produce tight bounds is when we are working with the Cartesian product of two complete graphs.  Suppose $m \geq n \geq 2$, and note that $K_m \square K_n$ is the line graph for the complete bipartite graph $K_{m,n}$.  So, by Galvin's celebrated result (\cite{G95}): $m=\chi(K_m \square K_n) = \chi_\ell(K_m \square K_n)$. However, Theorem~\ref{thm: Borow} only yields an upper bound of $\chi_\ell(K_m \square K_n) \leq m+n-1$ since $\chi_\ell(K_m)=\col(K_m)=m$.  Similarly, Theorem~\ref{thm: Brooks} only tells us $\chi_\ell(K_m \square K_n) \leq m+n-2$.

\par

Alon~\cite{A00} showed that for any graph $G$, $\col(G) \leq 2^{O(\chi_\ell(G))}$.  Combining this result with Theorem~\ref{thm: Borow} implies that we have an upper bound on $\chi_\ell(G \square H)$ in terms of only the list chromatic numbers of the factors.  Borowiecki et al.~\cite{BJ06} conjecture that a much stronger bound holds: there is a constant $A$ such that $\chi_\ell(G \square H) \leq A( \chi_\ell(G) + \chi_\ell(H))$.  While we will not address this conjecture in this paper, we will present results that are improvements on Theorem~\ref{thm: Borow} when the factors in the Cartesian product satisfy certain properties. Our main aim in this note is to illustrate how to utilize the classic Alon-Tarsi Theorem in these situations to get better bounds on the AT-number and consequently the list chromatic number.

\par

\subsection{Alon-Tarsi Number}

Suppose graph $D$ is a simple digraph. We say that $E$ is a \emph{circulation} contained in $D$ if $E$ is a spanning subgraph of $D$ and for each $v \in V(D)$, $d^-_E(v) = d^+_E(v)$ (Note: $d^-_E(v)$ represents the indegree of $v$ in $E$ and $d^+_E(v)$ represents the outdegree of $v$ in $E$).  We say a circulation is \emph{even} (resp. \emph{odd}) if it has an even (resp. odd) number of edges.  Classic results in graph theory tell us that a circulation is a digraph which consists of Eulerian components.  This means a circulation can be decomposed into directed cycles. Alon and Tarsi~\cite{AN92} use algebraic methods, subsequently called the Combinatorial Nullstellensatz, to obtain a remarkable relationship between a special orientation of a graph and a certain graph polynomial. The result is the celebrated Alon-Tarsi Theorem.
\begin{thm} [\textbf{Alon-Tarsi Theorem}] \label{thm: AT}
Let $D$ be an orientation of the simple graph $G$.  Suppose that $L$ is a list assignment for $G$ such that $|L(v)| \geq d^+_D(v) + 1$.  If the number of even and odd circulations contained in $D$ differ, then there is a proper $L$-coloring for $G$.  In addition, if the maximum indegree of $D$ is $m$ and the number of even and odd circulations contained in $D$ differ, then $\chi_\ell(G) \leq m+1$.
\end{thm}

\par

Recently, further implications of the Alon-Tarsi Theorem have appeared in the literature.  Before mentioning one of these implications we need some terminology.  For a graph $G$ suppose that for each $v \in V(G)$, $k$ \emph{tokens} are available at $v$.  Two players called the \emph{marker} and \emph{remover} then play the following game on the graph $G$.  For each round, the marker marks a non-empty subset, $M$, of vertices on the graph which uses one token for each marked vertex.  The remover then selects a subset of vertices, $I \subseteq M$, to remove such that $I$ is an independent set of vertices in $G$.  The marker wins by marking a vertex that has no tokens, and the remover wins by removing all the vertices from the graph.  A graph is said to be \emph{$k$-paintable} if the remover has a winning strategy when $k$ tokens are available at each vertex.  The \emph{paint number} or \emph{online choice number} of $G$, $\chi_p(G)$, is the smallest $k$ such that $G$ is $k$-paintable.  We have that $\chi_\ell(G) \leq \chi_p(G)$, and there exist graphs where $\chi_\ell(G) < \chi_p(G)$ (see~\cite{CL14} and~\cite{S09}).

\par

Schauz~\cite{S10} showed that if $D$ is an orientation of $G$ with maximum indegree $m$ such that $D$ satisfies the hypotheses of the Alon-Tarsi Theorem, then $\chi_p(G) \leq m+1$.  So, whenever we use the Alon-Tarsi theorem to bound the list chromatic number of a graph, we actually get the same bound on the paint number of the graph which is a stronger result.
\par

The \emph{Alon-Tarsi number} of $G$, $AT(G)$, is the smallest $k$ for which there is an orientation, $D$, of $G$ with max indegree $k-1$ such that the number of even and odd circulations contained in $D$ are different (i.e. the smallest $k$ for which the hypotheses of the Alon-Tarsi Theorem are satisfied). Jensen and Toft~\cite{JT95} first defined and suggested this graph invariant for study. Hefetz~\cite{H11} studied the AT-number and showed, among other results, that $AT(H) \leq AT(G)$ whenever $H$ is a subgraph of $G$. Recently, Zhu~\cite{Z17} showed that the AT-number of all planar graphs is at most 5, improving the classic list coloring bound on planar graphs.

\par

In summary we know that:
$$\chi(G) \leq \chi_\ell(G) \leq \chi_p(G) \leq AT(G).$$
In general all these inequalities can be strict, our aim is to find classes of graphs where they all are equal - a stronger form of chromatic-choosability. We study this phenomenon for Cartesian products of graphs.

%Since our focus in this paper is on chromatic choosability and improving upon the bounds on the list chromatic number in Theorems \ref{thm: Brooks} and \ref{thm: Borow}, we present all of our results in terms of bounds on the list chromatic number.  However, it is important to note that all of our results in Section~3 could be rephrased in terms of Alon-Tarsi choosability.

\section{Applying the Alon-Tarsi Theorem}\label{pathcycle}

In this section we will start by proving that the Cartesian product of an odd cycle and path is chromatic-choosable.  Then, we prove several generalizations of this argument.  Suppose that $G$ is the Cartesian product of an arbitrary odd cycle on $2k+1$ vertices and an arbitrary path on $n$ vertices.  Throughout the proof of Theorem~\ref{thm: oddcyclepath} we assume that we form $G$ as follows. Suppose we place the vertices of $C_{2k+1}$ around a circle so that two vertices are adjacent if and only if they appear consecutively along the circle, and we name the vertices in counterclockwise fashion around the circle as: $v_1, v_2, \ldots, v_{2k+1}$ (we also call this ordering of the vertices \emph{cyclic order}).  Similarly we name the vertices of $P_n$ as: $w_1, w_2, \ldots, w_n$ so that two vertices are adjacent if and only if they appear consecutively in this list (we say that this ordering of the vertices is \emph{in order}).

\begin{thm} \label{thm: oddcyclepath}
For any $k, n \in \N$, $AT(C_{2k+1} \square P_n)=3$. Consequently, $C_{2k+1} \square P_n$ is chromatic-choosable.
\end{thm}

\begin{proof}
The result is obvious when $n=1$.  So, we assume that $n \geq 2$. Now, we orient the edges of $G = C_{2k+1} \square P_n$ as follows. For each of the $n$ copies of $C_{2k+1}$ in $G$ we orient the edges of each copy in counterclockwise fashion.  Also we orient the edges of the form $\{(v_i, w_j) , (v_i, w_{j+1}) \}$ so that $(v_i, w_j)$ is the tail and $(v_i, w_{j+1})$ is the head (where $1 \leq i \leq 2k+1$ and $1 \leq j \leq n-1$).  We call the oriented version of $G$ digraph $D$.  We immediately note that for each $(v_i,w_j) \in V(D)$:
\[ d_{D}^-((v_i,w_j)) = \begin{cases}
      2 & \textrm{ if $i \neq 1$} \\
      1 & \textrm{ if $i=1.$} \\
   \end{cases} \]
Now, let $G^*$ be the graph obtained from $G$ by adding an extra edge in $G$ with endpoints $(v_1,w_1)$ and $(v_2, w_n)$.  Since $G$ is a subgraph of $G^*$, we have that $AT(G) \leq AT(G^*)$.  We will now show that $AT(G^*) \leq 3$.  To do this we form digraph $D^*$ by orienting the edges of $G^*$ so that all the edges of $G^*$ that are in $G$ are given the same orientation as in $D$ and the edge $\{(v_1,w_1), (v_2, w_n) \}$ is oriented so that $(v_1,w_1)$ is the head and $(v_2, w_n)$ is the tail.  We call this oriented edge $e^*$.  Note that for each $(v_i,w_j) \in V(D^*)$, $d_{D^*}^-((v_i,w_j)) \leq 2$.

\par

To show the number of even and odd circulations contained in $D^*$ differ, we will prove that the number of circulations contained in $D^*$ is odd.  We will refer to the $n$ oriented copies of $C_{2k+1}$ in $D^*$ as the \emph{base} \emph{cycles}.  We name the base cycles: $B_1, \ldots, B_n$ so that $B_i$ is the oriented copy of $C_{2k+1}$ such that all the vertices in $B_i$ have second coordinate $w_i$.  We note that the only directed cycles contained in $D^*$ are the base cycles and cycles that contain the edge $e^*$.  Let $\mathcal{C}$ be the set of all circulations in $D^*$.  We let
\[
  \mathcal{A} =\left\lbrace H \in \mathcal{C} \;\middle|\;
  \begin{tabular}{@{}l@{}}
  $H$ contains all the edges of at least one of the base cycles or\\
    $H$ does not include any edge from at least one of the base cycles
   \end{tabular}
  \right\rbrace.
\]
Suppose $H \in \mathcal{A}$.  We may assume $H$ contains all the edges from the base cycles: $B_{a_1} , \ldots, B_{a_m}$ and does not include any edges from the base cycles $B_{b_1}, \ldots, B_{b_t}$ (Note: at most one of these two lists of base cycles may be empty).  Now, we form another circulation $H' \in \mathcal{A}$ from $H$ as follows.  We delete all the edges of the base cycles: $B_{a_1} , \ldots, B_{a_m}$ and we add all the edges of the base cycles: $B_{b_1}, \ldots, B_{b_t}$.  We immediately note that $H \neq H'$ and this mapping gives us a way to pair up distinct elements of $\mathcal{A}$.  Thus, $|\mathcal{A}|$ is an even number.

\par

Now, let $\mathcal{B}= \mathcal{C}-\mathcal{A}$.  By the definition of $\mathcal{C}$ and $\mathcal{A}$ we have that $\mathcal{B}$ contains all the circulations of $D^*$ that contain at least one edge from each base cycle, but do not include all the edges of any base cycle.  Suppose that $H \in \mathcal{B}$.  We know that $H$ may be decomposed into directed cycles.  No directed cycle in the decomposition of $H$ can be a base cycle.  Since the only directed cycles in $D^*$ are the base cycles and cycles containing $e^*$, we may conclude that $H$ is a single cycle containing $e^*$ that contains at least one edge from each base cycle.  This means that we can view each element of $\mathcal{B}$ as a directed cycle starting with vertex $(v_1, w_1)$ and ending with the edge $e^*$.

\par

Now, for $q,r \in \Z$, $0 \leq q \leq n-2$, and $0 \leq r \leq 2k$, let $e_{(2k+1)q+r}$ be the directed edge with tail $(v_{r+1},w_{q+1})$ and head $(v_{r+1}, w_{q+2})$.  For $0 \leq i \leq n-2$, we let:
$$E_i = \{ e_{(2k+1)q+r} | \textrm{$q=i$ and $0 \leq r \leq 2k$} \}.$$
Intuitively, $E_i$ consists of the directed edges from the copies of $P_n$ that connect vertices in $B_{i+1}$ to vertices in $B_{i+2}$.  Now, suppose that $(e_{a_i})_{i=0}^{n-2}$ is a subsequence of the finite sequence of edges: $(e_i)_{i=0}^{(2k+1)(n-1)-1}$.  We call $(e_{a_i})_{i=0}^{n-2}$ a \emph{level} \emph{subsequence} \emph{of} \emph{edges} if $e_{a_i} \in E_i$ for each $i$, $a_1 \neq 0$, and $a_i \not \equiv a_{i+1} \; \textrm{mod} \; 2k+1$ for each $i \leq n-3$.  Now, let $Q$ be the set of all level subsequences of edges, and let $R$ be the set of all level subsequences of edges that do not have the edge $e_{(2k+1)(n-2)+1}$ as the last edge in the sequence.

\par

We will now construct a bijection between $\mathcal{B}$ and $R$.  Given a sequence of edges in $R$, $(e_{a_i})_{i=0}^{n-2}$, there is a unique cycle in $D^*$ starting at $(v_1, w_1)$ and ending with the edge $e^*$ that includes each edge in $(e_{a_i})_{i=0}^{n-2}$.  To form this cycle, simply follow a portion of the base cycle $B_1$ to get from $(v_1, w_1)$ to $e_{a_0}$.  The fact that $a_1 \neq 0$ guarantees that we must traverse at least one edge of $B_1$.  Then, for each $0 \leq i \leq n-3$, follow a portion of the base cycle $B_{i+2}$ to get from edge $e_{a_i}$ to $e_{a_{i+1}}$.  The fact that $a_i \not \equiv a_{i+1} \; \textrm{mod} \; 2k+1$ for each $i \leq n-3$ guarantees that we must traverse at least one edge of $B_{i+2}$.  Finally, follow a portion of the base cycle $B_n$ to get from $e_{a_{n-2}}$ to $e^*$.  The fact that $R$ contains all level subsequences that do not end with the edge incident to $e^*$ guarantees we must traverse at least one edge of $B_n$.  We immediately notice from this construction that the cycle we form contains at least one edge from each base cycle, but does not include all the edges of any base cycle.  Thus, the cycle we form is in $\mathcal{B}$.  So, we have a function from $R$ to $\mathcal{B}$.  To see that this function has an inverse suppose that $H \in \mathcal{B}$.  We know that we can view $H$ as  a directed cycle starting with $(v_1, w_1)$ and ending with the edge $e^*$.  By the way in which $D^*$ is oriented we know that $H$ must alternate between edge(s) from $B_i$ and an edge from $E_{i-1}$ for $1 \leq i \leq n-1$.  Then, cycle $H$ has edge(s) from $B_n$ followed by $e^*$.  Let $(u_k)_{k=0}^{n-2}$ be the ordered sequence of edges in $H$ that are in $\cup_{i=0}^{n-2} E_i$.  We immediately have that $u_k \in E_k$ for $0 \leq k \leq n-2$ by the way the paths are oriented.  Since $H$ contains at least one edge from each base cycle, $u_0 \neq e_0$, and if $u_k=e_{a'}$ and $u_{k+1}=e_{b'}$, then $a' \not \equiv b' \; \textrm{mod} \; 2k+1$ for $0 \leq k \leq n-3$.  So, $(u_k)_{k=0}^{n-2}$ is a level subsequence of edges by definition.  Finally, since $H$ contains at least one edge from $B_n$, we know that $u_{n-2} \neq e_{(2k+1)(n-2)+1}$.  So, $(u_k)_{k=0}^{n-2} \in R$.  We now have that our function maps $(u_k)_{k=0}^{n-2}$ to $H$ and the inverse maps $H$ to $(u_k)_{k=0}^{n-2}$.    Thus, we have a bijection between $\mathcal{B}$ and $R$ and we conclude that $|\mathcal{B}|=|R|$.

\par

Now, we will show that $|R|$ is an odd number.  We note that $|Q|=(2k)^{n-1}$ since forming an element of $Q$ leaves one with $2k$ distinct choices at each step.  We let $d_{r,n}$ be the number of elements in $Q$ that have a final edge with an index congruent to $r$ mod $2k+1$ (Note: $n$ is the number of vertices in our path and $0 \leq r \leq 2k$).  The following recursive relationships are immediate:
$$d_{r,n} = \sum_{i \in \{0, \ldots, 2k\}, i \neq r} d_{i, n-1} \; \; \textrm{for $n \geq 3$}$$
where $d_{0,2} = 0$ and $d_{i,2}=1$ for each $i \in \{1, \ldots 2k \}$.  An easy inductive argument shows that $d_{0,n}$ is even and $d_{i,n}$ is odd for each $i \in \{1, \ldots 2k \}$ whenever $n \geq 2$.  Since
$$|R|=|Q|-d_{1,n}=(2k)^{n-1}-d_{1,n},$$
we may conclude that $|R|$ is an odd number which immediately implies $|\mathcal{B}|$ is an odd number.

\par

Since $|\mathcal{C}|=|\mathcal{A}|+|\mathcal{B}|$, we have that $|\mathcal{C}|$ is odd.  Thus, $D^*$ contains an odd number of circulations.  This means that the number of even circulations in $D^*$ does not equal the number of odd circulations in $D^*$.  So $AT(G^*) \leq 3$ and hence $AT(G) \leq 3$. We also have that $3 \leq \chi(G)$.  So, $AT(G) = 3$, and we are done.
\end{proof}

{\bf Remark.} It is easy to notice that in our proof of Theorem~\ref{thm: oddcyclepath} we have proven something slightly stronger with regard to list coloring.  Specifically we proved that if we have a list assignment, $L$, for $G$ that assigns two colors to all but one vertex in the first base cycle and three colors to all other vertices of $G$, then there is a proper $L$-coloring for $G$.\\

\par
Having proven Theorem~\ref{thm: oddcyclepath}, it is easy to now classify the list chromatic number of the Cartesian product of an arbitrary cycle and path. Erd\H{o}s et al.~\cite{ET79} classified all graphs with list chromatic number equal to 2. Let $\Theta(l_1, \ldots, l_k)$ with branch vertices $u$ and $v$ be the graph that is the union of $k$ pairwise internally disjoint $u,v$-paths of lengths $l_1, \ldots, l_k$.

\begin{thm}[\cite{ET79}] \label{thm: 2choosable}
Let $G$ be a connected bipartite graph. Then, $\chi_\ell(G)=2$ if and only if $G$ has at most one cycle or the subgraph consisting of the non-cut-edges of $G$ is $\Theta(2,2,2t)$ for some $t \in \N$.
\end{thm}

This is the final ingredient in:

\begin{cor} \label{cor: classify}
For $k \in \N$:
\\
(i)  $\chi(C_{2k+1} \square P_n) = \chi_\ell(C_{2k+1} \square P_n) =3$ for $n \in \N$,
\\
(ii)  $\chi(C_{2k+2} \square P_1) = \chi_\ell(C_{2k+2} \square P_1) =2$, and $\chi_\ell(C_{2k+2} \square P_n) =3$ for $n \geq 2$.
\end{cor}

\begin{proof}  Statement (i) follows from Theorem~\ref{thm: oddcyclepath}. For statement (ii) notice that when $n \geq 2$, $C_{2k+2} \square P_n$ contains more than one cycle and no cut edges.  Moreover, all the vertices in $C_{2k+2} \square P_n$ have degree at least 3 which means that $C_{2k+2} \square P_n$ is not $\Theta(2,2,2t)$ for any $t \in \N$.  So, Theorem~\ref{thm: 2choosable} implies that $3 \leq \chi_\ell(C_{2k+2} \square P_n)$.  By Theorem~\ref{thm: 2choosable}, we also have that even cycles have list chromatic number equal to 2.  So, Theorem~\ref{thm: Borow} implies $\chi_\ell(C_{2k+2} \square P_n) \leq 2+2-1=3.$
\end{proof}

We now give a natural generalization of the argument of Theorem~\ref{thm: oddcyclepath}. 

\begin{thm} \label{cor: main result}
Suppose that $G$ is a complete graph or an odd cycle with $|V(G)| \geq 3$.  Suppose $H$ is a graph on at least two vertices that contains a Hamilton path, $w_1, w_2, \ldots, w_m$, such that $w_i$ has at most $k$ neighbors among $w_1, \ldots, w_{i-1}$. Then, $AT(G \square H) \leq \Delta(G)+k$.
\end{thm}

Before we prove this theorem a couple of remarks are worth making. We require $m \geq 2$, since we know $\chi_\ell(G \square H) = \Delta(G)+1$ when $m=1$.  Now, suppose $G$ and $H$ satisfy the hypotheses of the Theorem.  We have that $\chi_\ell(G) = \col(G) = \Delta(G)+1$ and $\col(H)-1 \leq k$.  Theorem~\ref{thm: Borow} tells us that $\chi_\ell(G \square H) \leq \Delta(G) + \chi_\ell(H)$, and Theorem~\ref{thm: Brooks} implies that $\chi_\ell(G \square H) \leq \Delta(G) + \Delta(H)$.  So, Theorem~\ref{cor: main result} gives us an improvement on these known bounds if and only if $k < \chi_\ell(H)$ and $k < \Delta(H)$.  It is easy to see that $k < \chi_\ell(H)$ and $k < \Delta(H)$ if and only if $k= \col(H)-1=\chi_\ell(H)-1$ and $d_H(w_m) \neq \Delta(H)$. We show examples in Section~3 where Theorem~\ref{cor: main result} improves these known bounds.  We now present the proof.

\begin{proof}
Suppose $G$ has $n$ vertices.  We name the vertices of $G$ (cyclically if $G$ is a cycle): $v_1, v_2, \ldots, v_n$. $G$ contains an odd cycle, $C$, as an induced subgraph. If $G$ is an odd cycle let $C=G$, and if $G$ is a complete graph let $C$ be the subgraph of $G$ induced by the vertices $v_{n-2}$, $v_{n-1}$, and $v_n$.  For the remainder of the proof assume that $C$ has $2k+1$ vertices where $k \in \N$ (Note: we know $k=1$ when $G$ is a complete graph).  We also let $P$ be the  Hamilton path, $w_1, w_2, \ldots, w_m$, contained in $H$.

\par

Now, consider the graph $G \square H$.  We form digraph $D$ from this graph by orienting its edges as follows.  We begin by orienting each of the $m$ copies of $C$ in a counterclockwise fashion.  Then, for each edge in a copy of $G$ and not in a copy of $C$ with endpoints $(v_r,w_u)$ and $(v_s,w_u)$ with $s > r$, we orient the edge so that $(v_r,w_u)$ is the tail and $(v_s,w_u)$ is the head.  Finally, for each edge in a copy of $H$ with endpoints $(v_u,w_r)$ and $(v_u,w_s)$ with $s > r$, we orient the edge so that $(v_u,w_r)$ is the tail and $(v_u,w_s)$ is the head.

\par

From $D$ we form the digraph $D^*$ by adding a directed edge with tail $(v_{n-2k+1}, w_m)$ and head $(v_{n-2k}, w_1)$.  Note that $D^*$ is a simple digraph since $m \geq 2$.  We will refer to the edge we added as $e^*$.  We immediately note that by the conditions placed on $G$ and $H$, we have that $d_{D^*}^- ((v_r,w_u)) \leq \Delta(G)-1+k$ for each $(v_r,w_u) \in V(D^*)$.  Similar to the proof of Theorem~\ref{thm: oddcyclepath}, we will now prove that the number of circulations in $D^*$ is odd.  We will refer to the $m$ oriented copies of $C$ in $D^*$ as the \emph{base} \emph{cycles}.  We name the base cycles: $B_1, \ldots, B_m$ so that $B_i$ is the oriented copy of $C$ such that all the vertices in $B_i$ have second coordinate $w_i$.  We note that the only cycles in $D^*$ are the base cycles and the cycles that contain $e^*$.  Let $S$ be the subgraph of $D^*$ that is made up of the $m$ oriented copies of $C$ in $D^*$, the oriented copies of $P$ that have first coordinate $v_{n-2k}$, $v_{n-2k+1}$, \ldots, and $v_{n}$, and the edge $e^*$.  Notice that $S$ is $C \square P$ plus edge $e^*$ oriented as it is in the proof of Theorem~\ref{thm: oddcyclepath}.  Let $\mathcal{C}$ be the set of all circulations in $D^*$. We define $\mathcal{A}$ as in the proof of Theorem~\ref{thm: oddcyclepath}.
%\[
%  \mathcal{A} =\left\lbrace H \in \mathcal{C} \;\middle|\;
%  \begin{tabular}{@{}l@{}}
%    $H$ contains all the edges of at least one of the base cycles or\\
%   $H$ does not include any edge from at least one of the base cycles
%   \end{tabular}
%  \right\rbrace.
%\]
By the same argument as in the proof of Theorem~\ref{thm: oddcyclepath}, we see that $|\mathcal{A}|$ is even.  Now, let $\mathcal{B}= \mathcal{C}-\mathcal{A}$.  As in the proof of Theorem~\ref{thm: oddcyclepath}, we have that any $K \in \mathcal{B}$ is a single cycle containing $e^*$ that contains at least one edge from each base cycle.  Moreover, any $K \in \mathcal{B}$ must be completely contained in $S$ (since any vertex of the form $(v_\ell,w_u)$ with $l<n-2k$ is in no cycles of $D^*$ and any cycle in $D^*$ that contains an oriented edge of a copy of $H$ that is not in a copy of $P$ must leave out at least one base cycle).

\par

Thus, by the same argument as in the proof of Theorem~\ref{thm: oddcyclepath}, we have that $|\mathcal{B}|$ is odd.  Thus, $|\mathcal{C}|$ is odd, and we have the desired result by the Alon-Tarsi Theorem.
\end{proof}

{\bf Remark.} Note that our proof also works when $G$ is any graph that contains an induced odd cycle $C$ such that every maximum degree vertex in $G$ is either in $C$ or adjacent to $C$.  However, this is not particularly useful for bounding $\chi_\ell(G \square H)$, since when $G$ is not an odd cycle or complete graph Theorems~\ref{thm: Brooks} and~\ref{thm: Borow} yield:
$$\chi_\ell(G \square H) \leq \chi_\ell(G) + \col(H) - 1 \leq \Delta(G) + k + 1 -1 = \Delta(G)+k.$$

\par
Now, we present another extension of Theorem~\ref{thm: oddcyclepath}.  We first need some definitions and notation.  Suppose that $G$ is an arbitrary graph, and $T$ is some subset of $V(G)$.  We write $G[T]$ for the subgraph of $G$ induced by the vertices in $T$.  Now, suppose that $G_1$ and $G_2$ are two arbitrary vertex disjoint graphs. The \emph{join} of the graphs $G_1$ and $G_2$, denoted $G_1 \vee G_2$, is the graph consisting of $G_1$, $G_2$, and additional edges added so that each vertex in $G_1$ is adjacent to each vertex in $G_2$.  When a graph, $G$, consists of $G_1$, $G_2$, and some set of additional edges (possibly empty) that have one endpoint in $V(G_1)$ and the other endpoint in $V(G_2)$ we say that $G$ is a \emph{partial join} of $G_1$ and $G_2$.

\begin{thm} \label{cor: joinpath}
Suppose that $G$ is a graph with the property that there exists a partition, $\{S_1, \ldots, S_m\}$, of $V(G)$ such that for each $i$ $G[S_i]$ is an odd cycle.  For each $i \geq 2$ suppose each vertex in $S_i$ has at most $\rho$ neighbors in $\cup_{j=1}^{i-1} S_j$ (we let $\rho = 0$ in the case that $m=1$). Then, $AT(G \square P_k) \leq 3+\rho$, for any $k \in \N$.
\end{thm}

Note that we obtain Theorem~\ref{thm: oddcyclepath} when $m=1$.  We also notice that when $m=2$, $G$ is a partial join of two odd cycles.  When $m=3$, $G$ is a partial join of the odd cycle $G[S_3]$ and $G[S_1 \cup S_2]$ (where $G[S_1 \cup S_2]$ is itself the partial join of two odd cycles). And so on.

\begin{proof}
First, for each $i$, we index the vertices of $G[S_i]$ in cyclic order as: $v_{i,1}, v_{i,2}, \ldots, v_{i,m_i}$ where we know that $m_i$ is odd for each $i$.  If $e \in E(G)$ is an edge with one endpoint in $S_i$ and one endpoint in $S_j$ with $i \neq j$ we refer to $e$ as a \emph{connecting edge}.

\par

We now turn our attention to the case where $k=1$.  In this case, we must show that $AT(G) \leq 3+\rho$.  We form an oriented version, $D$, of the graph $G$ as follows.  For each odd cycle, $G[S_i]$, orient the edges of the cycle so that $v_{i,l}$ is the tail and $v_{i,l+1}$ is the head for $1 \leq l \leq m_i - 1$.  Then, orient the final edge of the cycle so that $v_{i,1}$ is the tail and $v_{i,m_i}$ is the head.  Finally, for each connecting edge, $e \in E(G)$, with one endpoint in $S_i$ and one endpoint in $S_j$ with $i < j$ orient $e$ so that its tail is in $S_i$ and its head is in $S_j$.  One may now note that for each $v \in S_i$,
$$d_{D}^- (v)  \leq 2 + \rho.$$
\par We now claim that $D$ is acyclic.  First, note that no oriented connecting edge can be in a cycle in $D$ since there is no way in $D$ to get from a vertex in $S_j$ to a vertex in $S_i$ when $i < j$.  This means that if there is a cycle, $C$, in $D$ there must exist an $i$ such that the vertices of $C$ are a subset of the vertices of $S_i$.  However, $D[S_i]$ is acyclic by construction.  So, no such $C$ can exist.  This means that $D$ has one even circulation (the circulation with no edges) and zero odd circulations.  Then $AT(G) \leq 3 + \rho$ as desired.
\\

\par We now turn our attention to the case where $k \geq 2$. Let $H=P_k$. We note that $G \square H$ is made of the disjoint union of graphs:
$$ \sum_{i=1}^m (G[S_i] \square H) $$
plus the connecting edges in each of the $k$ copies of $G$.  Let $G^{**}$ be the graph formed from $G \square H$ by adding an edge in $G[S_i] \square H$ for each $i$ as we did in the proof of Theorem~\ref{thm: oddcyclepath} to form $G^*$.  We call the newly formed subgraph of $G^{**}$ consisting of $G[S_i] \square H$ plus an additional edge, $M_i$ for each $i$.  We form an oriented version, $D$, of $G^{**}$ as follows.  For each subgraph $M_i$ of $G^{**}$, we orient this subgraph just as we oriented $D^*$ in the proof of Theorem~\ref{thm: oddcyclepath}.  Finally, we orient all the connecting edges in all the copies of $G$ just as we did in the case of $k=1$.  One may note that for each $v \in V(D)$, we have that $d_{D}^- (v) \leq 2 + \rho$.  Similar to the case where $k=1$, we note that no connecting edge from any copy of $G$ is in a cycle contained in $D$.  This means that if $C_i$ is the number of circulations in the oriented version of $M_i$, the number of circulations in $D$ is equal to
$$ \prod_{i=1}^m C_i$$
which by the proof of Theorem~\ref{thm: oddcyclepath} is an odd number.  Thus, the number of even circulations in $D$ does not equal the number of odd circulations in $D$, and we have that $AT(G \square H) \leq AT(G^{**}) \leq \rho + 3.$
\end{proof}

Finally, we can easily combine the idea of the proofs of Theorems~\ref{cor: main result} and~\ref{cor: joinpath} to obtain the following.

\begin{thm} \label{cor: generaljoin}
Suppose that $G$ is a graph with the property that there exists a partition, $\{S_1, \ldots, S_m\}$, of $V(G)$ such that for each $i$ $|S_i| \geq 3$ and $G[S_i]$ is an odd cycle or a complete graph. For each $i \geq 2$ suppose each vertex in $S_i$ has at most $\rho_i$ neighbors in $\cup_{j=1}^{i-1} S_j$, and let $\rho_1=0$.  For each $i \geq 1$ we let:
\[ \alpha_i = \begin{cases}
      \rho_i+3 & \textrm{ if $G[S_i]$ is an odd cycle} \\
      \rho_i + |S_i| & \textrm{ if $G[S_i]$ is a complete graph.} \\
   \end{cases} \]
Now, let $\alpha = \max_{i} \alpha_i$.  Suppose $H$ is a graph on at least two vertices that contains a Hamilton path, $w_1, w_2, \ldots, w_n$, such that $w_i$ has at most $k$ neighbors among $w_1, \ldots, w_{i-1}$.  Then, $AT(G) \leq \alpha$ and $AT (G \square H) \leq \alpha+k-1$.
\end{thm}

Both Theorems~\ref{cor: joinpath} and~\ref{cor: generaljoin} are sharp and give improvements over existing bounds as shown in examples in the next section.

\section{Some Examples}\label{examples}

In this section we present some examples where our results from Section~2 improve upon known bounds for the list chromatic number. We let the \emph{kth power} of graph $G$, denoted $G^k$, be the graph with vertex set $V(G)$ where two vertices are adjacent if their distance in $G$ is at most $k$. It is easy to see that $P_n^r$, $1 \leq r \leq n-1$, satisfies $\chi(P_n^r) = \chi_\ell(P_n^r) = \col(P_n^r) = r+1$. We already know Theorem~\ref{cor: main result} is sharp by $r=1$ below.
\begin{cor} \label{cor: powerofpath}
Suppose that $k,r, n \in \N$ are such that $n \geq 2$ and $r \leq n-1$.\\ Then, $\max \{r+1,3\} \leq AT(C_{2k+1} \square P_n^r) \leq r + 2$.
\end{cor}

\begin{proof}
Let $G=C_{2k+1}$ and $H=P_n^r$. Name the vertices of $H$: $w_1, w_2, \ldots, w_n$ so that $w_i$ is the $i^{th}$ vertex of the underlying path on $n$ vertices contained in $H$.  Clearly, $H$ contains the Hamilton path: $w_1, w_2, \ldots, w_n$, and for each $i$, $w_i$ has at most $r$ neighbors among $w_1, w_2, \ldots, w_{i-1}$.  Thus, by Theorem~\ref{cor: main result}, we have that $AT(G \square H) \leq 2+r$.
\end{proof}

We note that Theorem~\ref{thm: Borow} only yields that $\chi_\ell(G \square H) \leq r + 3$.  Also, Theorem~\ref{thm: Brooks} yields $\chi_\ell(G \square H) \leq 2 + \Delta(P_n^r)$, and $r+1 \leq \Delta(P_n^r)$ when $n \geq 3$ and $r \leq n-2$.  So, the above example improves upon known bounds on the list chromatic number when $n \geq 3$ and $r \leq n-2$.  We suspect that $C_{2k+1} \square P_n^r$ is often chromatic-choosable, but improving upon our upper bound with the Alon-Tarsi Theorem seems difficult.  For example, it would be impossible to find an orientation of $C_{2k+1} \square P_n^2$ with max indegree of 2 for large values of $n$.
\par
Before we move on to an example for Theorems~\ref{cor: joinpath} and~\ref{cor: generaljoin}, let us note a couple of other facts implied by Theorem~\ref{cor: main result}: $AT(K_n \square P_m)=n$, for $n \ge 3$; and $n-2 \leq AT(C_{2k+1} \square H_n) \leq n-1$ where $H_n=K_n-E(P_4)$ for $n\ge 5$.

\begin{cor} \label{cor: join1}
For $G = K_m \vee C_{2k+1}$, $AT(G \square P_n) = m+3$, where $m, k, n \in \N$. Consequently, $(K_m \vee C_{2k+1}) \square P_n$ is chromatic-choosable.
\end{cor}

\begin{proof}
The result is obvious when $n=1$ since $\col(K_m \vee C_{2k+1}) = m+3$.  So, suppose $n \geq 2$.  Let $G_1 = K_m$ and $G_2= C_{2k+1}$. Since $\chi(G_1 \vee G_2) = \chi(G_1) + \chi(G_2)$, we have that $m+3 = \chi(G) = \chi(G \square P_n)$.
\par
First, consider the case where $m \geq 3$.  Let $\{S_1, S_2 \}$ be the partition of $V(G)$ where $S_1=V(G_1)$ and $S_2=V(G_2)$.  Note $G[S_1]$ is a complete graph, and $G[S_2]$ is an odd cycle.  Also, each vertex in $S_2$ has exactly $m$ neighbors in $S_1$.  So, by Theorem~\ref{cor: generaljoin} we have that $AT(G \square P_n) \leq m+3$.
\par
Finally, suppose $m=1,2$.  Let $G^{(m)}$ be a partial join of $G_3=C_{3}$ and $G_2$ so that $m$ vertices in $V(G_{3})$ are adjacent to all the vertices in $V(G_2)$, and the other $3-m$ vertices in $V(G_3)$ are not adjacent to any of the vertices in $V(G_2)$.  Now, let $\{S_3, S_4 \}$ be the partition of $V(G^{(m)})$ where $S_3=V(G_3)$ and $S_4=V(G_2)$. Note $G^{(m)}[S_3]$ and $G^{(m)}[S_4]$ are odd cycles, and each vertex in $S_4$ has exactly $m$ neighbors in $S_3$.  So, by Theorem~\ref{cor: joinpath} we have that $AT(G^{(m)} \square P_n) \leq m+3$.  The result follows since $G \square P_n$ is a subgraph of $G^{(m)} \square P_n$.
\end{proof}

We note that when $n \geq 2$ Theorem~\ref{thm: Borow} only yields that $\chi_\ell((K_m \vee C_{2k+1}) \square P_n) \leq m+4$.  Also, when $n \geq 3$, Theorem~\ref{thm: Brooks} only tells us $\chi_\ell((K_m \vee C_{2k+1}) \square P_n) \leq \max \{m+4, m + 2k + 2\}$.

There is still much to be discovered about the list chromatic number of the Cartesian product of graphs.  Aside from the important conjectures proposed in~\cite{BJ06}, this paper also provides us with some interesting questions. Specifically, an ambitious question would be: Can we determine when $G \square H$ will be chromatic-choosable based upon some property of the factors? We further explore this question in a following paper~\cite{KM2}. Simpler questions, motivated by Corollaries~\ref{cor: classify} and~\ref{cor: powerofpath}, include:  For what graphs $G$, is $G \square P_n$ chromatic-choosable? When is $C_{2k+1} \square P_n^r$ chromatic-choosable?  For the first of these questions one may conjecture based upon the results of this paper that if $G$ is chromatic-choosable and $\chi(G) \geq 3$, then $G \square P_n$ chromatic-choosable.  However, this conjecture is false since one can construct a 3-assignment to show that $3< \chi_\ell(C_6^2 \square P_2)$.

\end{document}